\documentclass[letterpaper,10pt]{article}
\usepackage{amsmath}
\usepackage{amsfonts}
\usepackage{amssymb}
\usepackage{amsthm}
\usepackage{bbm}

\numberwithin{equation}{section}

\newtheorem{thm}[equation]{Theorem}
\newtheorem{lem}[equation]{Lemma}
\newtheorem{coro}[equation]{Corollary}

\theoremstyle{definition}

\theoremstyle{remark}
\newtheorem{remark}[equation]{Remark}

\allowdisplaybreaks[2]

\title{Small $\kappa$ Asymptotics of the Almost Sure Lyapunov Exponent for the Continuum Parabolic Anderson Model}
\author{Michael Rael\footnote{Department of Mathematics, University of California, Irvine, CA 92697-3875, USA}}
\date{\today}

\begin{document}

\maketitle

\begin{abstract}
We prove that the almost sure Lyapunov exponent $\lambda(\kappa)$ of the continuous space Parabolic Anderson Model is bounded above by $c_u \kappa^{1/3}$ as $\kappa\downarrow0$ under mild regularity conditions.
This bound of the same order of the previously proven lower bound, $\lambda(\kappa) \ge c_l \kappa^{1/3}$.
\end{abstract}

\section{Background}

Let $\{ W_x : x\in\mathbb{R}^d \}$ be a Gaussian field of identically distributed copies of mean $0$ Brownian Motion defined on the probability space $(\Omega,\mathcal{F},Q)$.
This field has covariance given by $\mathbb{E}^Q[W_x(t) W_y(s)] = \Gamma(x-y) (t \wedge s)$ where $\Gamma(z) = \Gamma(\|z\|_2)$ is twice continuously differentiable,
  bounded by $0\le\Gamma(z)\le1$, and has the following Taylor expansion near $0$:
\begin{equation} \label{gamma_taylor}
  \Gamma(z) = 1 - c_d \|z\|_2^2 + o(\|z\|_2^2) .
\end{equation}
This assumption on the Taylor expansion of $\Gamma$ can be relaxed considerably, see Remark \ref{metric_domination}.

We consider the following stochastic differential equation over $\mathbb{R}^d$,
\begin{equation} \label{PAM_S}
  du(x,t) =  \frac{\kappa}{2} \triangle u(x,t)dt + u(x,t) \partial W_x(t), \quad x\in\mathbb{R}^d, t>0,
\end{equation}
where $\kappa > 0$ is constant, $\partial W_x$ denotes the Stratonovich differential of $W_x$, $\triangle$ is the Laplacian, and $u(x,0)\equiv1$.
Equation (\ref{PAM_S}) is called the Parabolic Anderson Model in $\mathbb{R}^d$, hereafter PAM.

In \cite{CaVi} the existence of a solution to (\ref{PAM_S}) was established, as was the validity of the Feynman-Kac representation of the solution:
\begin{equation} \label{F-K_S}
  u(x,t) = \mathbb{E}^X_x \left[ e^{\int_0^t dW_{X(t-s)}(s)} \right].
\end{equation}
where $X(s)$ is a $\kappa$ speed $d$-dimensional Brownian Motion, i.e.\ the diffusion with generator $\frac{\kappa}{2}\triangle$.
Throughout this paper $\mathbb{P}$ and $\mathbb{E}^X$ denote the probability measure of $X$ and expectation with respect to $\mathbb{P}$, respectively.

In studying the PAM, the Lyapunov exponent
\begin{equation} \label{lyapunov}
  \lambda(\kappa) = \lim_{t\to\infty} \frac{1}{t} u(x,t)
\end{equation}
has been of primary interest.
The existence of $\lambda(\kappa)$ as a deterministic limit, and its convexity were established in \cite{CaMo, CrMoSh, CrMo06}.

It it the purpose pf this paper to improve previously derived bounds on the small $\kappa$ behavior of $\lambda(\kappa)$.
In \cite{CrMo06} it was proven that 
\begin{equation} \label{one_third_lower_bound}
  \liminf_{\kappa \downarrow 0} \frac{\lambda(\kappa)}{\kappa^{1/3}} \ge c
\end{equation}
and that
\begin{equation}
  \limsup_{\kappa \downarrow 0} \frac{\lambda(\kappa)}{\kappa^{1/5}} \le c.
\end{equation}
It was conjectured that the lower bound (\ref{one_third_lower_bound}) gave the correct asymptotics for $\lambda(\kappa)$.
We prove this conjecture.

\begin{thm} \label{one_third_bound}
  Under the aforementioned conditions
  \begin{equation}
    \limsup_{\kappa \downarrow 0} \frac{\lambda(\kappa)}{\kappa^{1/3}} \le c.
  \end{equation}
\end{thm}

This result is notable in that it is one of the rare examples of differing behavior in the PAM and the discrete PAM, $x\in\mathbb{Z}^d$.
In the discrete PAM,
\begin{equation*}
  \lim_{\kappa \downarrow 0} \lambda(\kappa) \ln(1/\kappa) = c
\end{equation*}
as proven in \cite{CaViMo, CaMo, CrMoSh}.
For information of the discrete PAM see \cite{CaMo, CrMoSh}.

\section{Proof of Theorem \ref{one_third_bound}}

\begin{remark}
  This approximation approach follows from an idea of Michael Cranston's,
    who used the same approximating functions in an unpublished proof of the lower bound (\ref{one_third_lower_bound}).
\end{remark}

For convenience let
\begin{equation}
  F(f) = \int_0^t dW_{f(t-s)}(s) .
\end{equation}

We will approximate the Brownian paths in the Feynman-Kac representation of $u(0,t)$ using Cameron-Martin functions.
In particular, we work with the families of the form
\begin{equation}
  H_t(C) = \left\{ f \in C_0([0,t]; \mathbb{R}^d), \| f' \|_2 \le C \right\}.
\end{equation}

We have a topology on $C_0([0,t]; \mathbb{R}^d)$ defined by the natural metric, $$d(f,g) = \left( \mathbb{E}^Q (F(f)-F(g))^2 \right)^{1/2}.$$
We approximate $u(x,t)$ by the contribution from successively larger balls in this topology.
Defining the increasing sequences
\begin{align}\label{C_and_epsilon_def}
  C_n & = M n \kappa^{2/3} t^{1/2}, & \epsilon_n & = n^{1/2} \kappa^{1/6} t^{1/2}
\end{align}
we let $\Gamma_n$ be a minimal $\epsilon_n/2$-net of $H_t(C_n)$.
It will follow from Lemma \ref{ball_exclusion_bound} that
\begin{equation*}
  \lim_{n\to\infty} \mathbb{P}\left( d(X,\Gamma_n) > \epsilon_n \right) \to 0.
\end{equation*}

Adopting the shorthand
\begin{equation}\label{def_Efn}
  E_{f,n} = \left\{ d(X,f)\le\epsilon_n, d(X,\Gamma_{n-1})>\epsilon_{n-1} \right\},
\end{equation}
we bound the Feynman-Kac formula (\ref{F-K_S}) of $u(x,t)$ though the following decomposition
\begin{align}
  u(x,t) & = \mathbb{E}^X_x \left[ e^{F(X)} \right] \nonumber\\
      & = \sum_{n\ge1} \mathbb{E}^X_x \left[ e^{F(X)} ; d(X,\Gamma_n)\le\epsilon_n, d(X,\Gamma_{n-1})>\epsilon_{n-1} \right] \nonumber\\
      & \le \sum_{n\ge1} \sum_{f\in\Gamma_n} e^{F(f)} \mathbb{E}^X_x \left[ e^{F(X)-F(f)} ; E_{f,n} \right] \nonumber\\
      & = \sum_{n\ge1} \sum_{f\in\Gamma_n} e^{F(f)} \mathbb{E}^X_x \left[ \left. e^{F(X)-F(f)} \right| E_{f,n} \right] \mathbb{P}( E_{f,n} ) \nonumber\\
      & \le \sum_{n\ge1} \sum_{f\in\Gamma_n} e^{F(f)} \mathbb{E}^X_x \left[ \left. e^{F(X)-F(f)} \right| E_{f,n} \right] \mathbb{P}( d(X,\Gamma_{n-1})>\epsilon_{n-1} ) \nonumber\\
      & \le \sum_{n\ge1} |\Gamma_n| \left( \sup_{f\in\Gamma_n} e^{F(f)} \right) \left( \sup_{f\in\Gamma_n} \mathbb{E}^X_x \left[ \left. e^{F(X)-F(f)} \right| E_{f,n} \right] \right)
            \nonumber\\ & \qquad\qquad \cdot \mathbb{P}( d(X,\Gamma_{n-1})>\epsilon_{n-1} ) . \label{decomp1}
      %\\& \le \sum_{n\ge1} |\Gamma_n| \left( \sup_{f\in\Gamma_n} e^{F(f)} \right) \left( \sup_{f\in\Gamma_n} \mathbb{E}^X_x \left[ \left. e^{F(X)-F(f)} \right| X\in B_{\epsilon_n}(f) \setminus B_{\epsilon_{n-1}}(\Gamma_{n-1}) \right] \right) \mathbb{P}( d(X,\Gamma_{n-1})>\epsilon_{n-1} )
\end{align}

To show that $\lambda(\kappa) = O(\kappa^{1/3})$ we must show that $\limsup_{t\to\infty} u(x,t) \le c' e^{c \kappa^{1/3} t}$ for $\kappa\in(0,\kappa_0)$, $\kappa_0$ small.
In light of the elementary fact that
\begin{equation}\label{exp_sum}
  \sum_{k\ge1} e^{-\alpha k} = -1 + \frac{1}{1-e^{-\alpha}} = \frac{e^{-\alpha}}{1-e^{-\alpha}} \le e^{-\alpha}
\end{equation}
for $\alpha\ge0$ it suffices to show that each of the summands in (\ref{decomp1}) are $\le c' e^{c n \kappa^{1/3} t}$.
We proceed with series of lemmata, which when taken together establish such a bound.

First is an entropy bound on $H_t(C)$.
$\mathcal{N}_\epsilon(H_t(C))$ denotes the number of $\epsilon$-balls under the $d$ metric needed to cover $H_t(C)$.
\begin{lem} \label{K-T_bound}
  \begin{equation*}
    \mathcal{N}_\epsilon(H_t(C)) \le c_1 \exp\left\{c_2 \frac{Ct}{\epsilon}\right\} .
  \end{equation*}
\end{lem}
\begin{proof}
  From \cite{KoTi} we have that
  \begin{equation} \label{K-T_orig}
    \mathcal{N}'_\epsilon(H_{\pi/2}(C)) \le c_1 \exp\left\{c_2' \frac{C}{\epsilon}\right\} 
  \end{equation}
  where $\mathcal{N}'_\epsilon(H_t(C))$denotes the number of $\epsilon$-balls under the $L^2$ metric needed to cover $H_t(C)$.
  We will use scaling relations to derive the lemma.
  
  For $f \in H_{\pi/2}(C)$ let $g(s) = f\left(\frac{\pi}{2t}s\right)$, $g \in C_0([0,t])$.
  Then
  \begin{equation*}
    \| g' \|_2^2 = \int_0^t \left\| \frac{\pi}{2t}  f'\left(\frac{\pi}{2t}s\right)\right\|_2^2 ds = \frac{\pi}{2t} \int_0^{\pi/2} \|f'(r)\|_2^2 dr
        = \frac{\pi}{2t} \| f' \|_2^2 .
  \end{equation*}
  Thus $g \in H_t\left(\sqrt{\frac{\pi}{2t}}C\right)$ and we have a bijection $H_{\pi/2}(C) \leftrightarrow H_t\left(\sqrt{\frac{\pi}{2t}}C\right)$.
  This mapping also affects the radii of $L^2$-balls.
  Letting $h(s) = k\left(\frac{\pi}{2t}s\right)$ where $k \in H_{\pi/2}(C)$ wee see that
  \begin{equation*}
    \int_0^t \|g(s)-h(s)\|_2^2 ds = \int_0^t \left\| f\left(\frac{\pi}{2t}s\right) - k\left(\frac{\pi}{2t}s\right) \right\|_2^2 ds = \frac{2t}{\pi} \int_0^t \| f(r) - k(r) \|_2^2 dr .
  \end{equation*}
  So a $L^2$-ball of radius $\epsilon$ in $H_{\pi/2}(C)$ maps to a $L^2$-ball of radius $\sqrt{\frac{2t}{\pi}}\epsilon$ in $H_t\left(\sqrt{\frac{\pi}{2t}}C\right)$.
  From (\ref{K-T_orig}) and these scaling arguments that
  \begin{align*}
    \mathcal{N}'_\epsilon(H_t(C)) & = \mathcal{N}'_{\sqrt{\frac{\pi}{2t}} \epsilon}\left( H_{\pi/2}\left(\sqrt{\frac{2t}{\pi}} C\right) \right) \\
        & \le c_1 \exp\left\{c_2' \frac{\sqrt{\frac{2t}{\pi}} C}{\sqrt{\frac{\pi}{2t}} \epsilon}\right\} = c_1 \exp\left\{c_2 \frac{Ct}{\epsilon}\right\} .
  \end{align*}

  This bound has so far been proven for the $L^2$ metric, we need to show that it applies to the $d$ metric.
  It follows from (\ref{gamma_taylor}) that
  \begin{align*}
    d^2(f,g) & = \mathbb{E}_Q (F(f)-F(g))^2 \\
        & = 2 \int_0^t 1 - \Gamma(f(s)-g(s)) ds \\
        & = 2 \int_0^t c_d \|f(s)-g(s)\|_2^2 - o\left(\|f(s)-g(s)\|_2^2\right) ds \\
        & \le 2 c_d \|f-g\|_2^2.
  \end{align*}
  Thus every radius $\epsilon$ $d$-ball is contained in a radius $\epsilon\sqrt{2 c_d}$ $L^2$-ball and, allowing for changes to the constant $c_2$, we have proven the bound.
\end{proof}

\begin{remark}\label{metric_domination}
  The domination of the $d$ metric by the $L^2$ metric in the final step of the proof of Theorem \ref{K-T_bound} is the sole reason for the assumption (\ref{gamma_taylor}).
  This assumption can be weakened, so long as the metric domination is preserved.
\end{remark}

\begin{coro}\label{Gamma_n_bound}
\begin{equation*}
  |\Gamma_n| \le c_1 e^{ 2c_2 M n \kappa^{1/3} t} .
\end{equation*}
\end{coro}
\begin{proof}
  \begin{align*}
    |\Gamma_n| & = \mathcal{N}_{\epsilon_n/2}(H_t(C_n)) \\
        & \le c_1 \exp\left\{2c_2 \frac{C_n t}{\epsilon_n}\right\}\\
        & = c_1 e^{ 2c_2 M n^{1/2} \kappa^{1/2} t}\\
        & \le c_1 e^{ 2c_2 M n \kappa^{1/3} t} .
  \end{align*}
  Note that the last inequality follows from $\kappa<1$.
\end{proof}

\begin{lem} \label{F-T}
  \begin{equation*}
    \mathbb{E}_Q \sup_{f \in H_t(C)} F(f) \le 2 c_4 {(c_3 C)}^{1/2} t^{3/4} + O\left( t^{1/2} \right)
  \end{equation*}
\end{lem}
\begin{proof}
  First we apply Fernique-Talagrand (Theorem 4.1 of \cite{Ad90}) and then we use of Lemma \ref{K-T_bound}
    and the elementary fact that $d(f,g)\le(2t)^{1/2}$ to obtain
  \begin{align*}
    \mathbb{E}_Q \sup_{f \in H_t(C)} F(f) & \le K \int_0^{diam(H_t(C))} \left( \ln \mathcal{N}_\delta(H_t(C)) \right)^{1/2} d\delta \\
        & \le K \int_0^{(2t)^{1/2}} \left( \ln c_1 + c_2 \frac{C t}{\delta} \right)^{1/2} d\delta \\
        &  = K (\ln c_1)^{1/2} \int_0^{(2t)^{1/2}} \left( 1 + \frac{c_2 C t}{\ln c_1} \frac1\delta \right)^{1/2} d\delta .
  \end{align*}
  To proceed we first make the substitution $r = \frac{\delta \ln c_1}{c_2 C t}$ so that
  \begin{equation*}
    \mathbb{E}_Q \sup_{f \in H_t(C)} F(f) \le K \frac{c_2 C t}{(\ln c_1)^{1/2}} \int_0^{\frac{\sqrt{2} \ln c_1}{c_2 C t^{1/2}}} \left( 1 + \frac1r \right)^{1/2} dr .
  \end{equation*}
  For brevity, we define the constants
  \begin{align*}
    c_3 & = \frac{\sqrt{2} \ln c_1}{c_2} & c_4 & = \frac{K c_2}{(\ln c_1)^{1/2}} .
  \end{align*}
  Then we make the trigonometric substitution $\tan\theta = r^{1/2}$.
  Thus
  \begin{align*}
    \mathbb{E}_Q \sup_{f \in H_t(C)} F(f) & \le c_4 C t \int 2 \sec^3 \theta d\theta \\
        & = c_4 C t \left[ \tan\theta \sec\theta + \ln\left| \tan\theta + \sec\theta \right| \right] \\
        & = c_4 C t \left[ \sqrt{r}\sqrt{r+1} + \ln\left| \sqrt{r} + \sqrt{r+1} \right| \right]_0^{\frac{c_3}{Ct^{1/2}}} \\
        & = c_4 C t \left[ \sqrt{\frac{c_3}{Ct^{1/2}}}\sqrt{\frac{c_3}{Ct^{1/2}}+1} \right. \\
              & \qquad\qquad + \left. \ln\left| \sqrt{\frac{c_3}{Ct^{1/2}}} + \sqrt{\frac{c_3}{Ct^{1/2}}+1} \right| \right].
  \end{align*}
  Now we make use of Taylor's Theorem. First for $\sqrt{z+1} = 1 + \frac{z}{2} + O(z^2)$ and then for $\ln(1+z) = z + O(z^2)$.
  \begin{align*}
    \mathbb{E}_Q \sup_{f \in H_t(C)} F(f) & \le c_4 C t \left[ \sqrt{\frac{c_3}{Ct^{1/2}}}\left( 1 + \frac{c_3}{2Ct^{1/2}} + O\left( \frac{1}{C^2t} \right) \right) \right. \\
              & \qquad\qquad + \left. \ln\left| \sqrt{\frac{c_3}{Ct^{1/2}}} + 1 + \frac{c_3}{2Ct^{1/2}} + O\left( \frac{1}{C^2t} \right) \right| \right] \\
        & = c_4 C t \left[ \sqrt{\frac{c_3}{Ct^{1/2}}}\left( 1 + \frac{c_3}{2Ct^{1/2}} + O\left( \frac{1}{C^2t} \right) \right) \right. \\
              & \qquad\qquad + \left. \sqrt{\frac{c_3}{Ct^{1/2}}} + \frac{c_3}{2Ct^{1/2}} + O\left( \frac{1}{Ct^{1/2}} \right) \right] \\
        & = 2 c_4 {(c_3 C)}^{1/2} t^{3/4} + O\left( t^{1/2} \right) . \qedhere
  \end{align*}
\end{proof}

\begin{coro}\label{F-T_bound}
  For all $n\in\mathbb{N}$,
  \begin{equation*}
    \sup_{f\in\Gamma_n} e^{F(f)} \le e^{4 c_4 c_3^{1/2} M^{1/2} n \kappa^{1/3} t} \text{ $Q$-a.s.\ as } t\to\infty
  \end{equation*}
  for $t\in\mathbb{N}$.
\end{coro}
\begin{proof}
  We again note that $d(f,g)\le(2t)^{1/2}$ and then applying Borell's Inequality (Theorem 2.1 in \cite{Ad90}),
  \begin{align*}
    & Q\left( \sup_{f \in \Gamma_n} e^{F(f)} > e^{4 c_4 c_3^{1/2} M^{1/2} n \kappa^{1/3} t} \right) \\
        & \qquad \le Q\left( \sup_{f \in H_t(C_n)} e^{F(f)} > e^{4 c_4 c_3^{1/2} M^{1/2} n \kappa^{1/3} t} \right) \\
        & \qquad \le Q\left( \sup_{f \in H_t(C_n)} F(f) > 4 c_4 c_3^{1/2} M^{1/2} n \kappa^{1/3} t \right) \\
        & \qquad \le 2 \exp\left\{ \frac{-1}{2t} \left( 4 c_4 c_3^{1/2} M^{1/2} n \kappa^{1/3} t - 2 c_4 {(c_3 C_n)}^{1/2} t^{3/4} - O\left( t^{1/2} \right) \right)^2 \right\} \\
        & \qquad = 2 \exp\left\{ \frac{-1}{2t} \left( 4 c_4 c_3^{1/2} M^{1/2} n \kappa^{1/3} t - 2 c_4 {(c_3 M n)}^{1/2} \kappa^{1/3} t - O\left( t^{1/2} \right) \right)^2 \right\} \\
        & \qquad \le \exp\left\{ \frac{-nt}{2} \left( 2 c_4 c_3^{1/2} M^{1/2} \kappa^{1/3} - O\left( t^{-1/2} \right) \right)^2 \right\} \\
        & \qquad \le \exp\left\{ \frac{-nt}{2} \left( c_4 c_3^{1/2} M^{1/2} \kappa^{1/3} \right)^2 \right\} \text{ for $t>T$,}
  \end{align*}
  where the constant $T$ is taken to be large enough that this holds for all $n\in\mathbb{N}$ and $\kappa\in(0,\kappa_0)$.
  Summing over $n$ and then $t\in\{i\in\mathbb{N}, i>T\}$ using (\ref{exp_sum}) we see that this quantity is summable.
  An application of the Borell-Cantelli Lemma competes the proof.
\end{proof}

\begin{lem}\label{X-f_bound}
  For all $n\in\mathbb{N}$,
  \begin{equation*}
    \sup_{f\in\Gamma_n} \mathbb{E}^X_x \left[ \left. e^{F(X)-F(f)} \right| E_{f,n} \right]
        \le e^{(1+4c_2M)n\kappa^{1/3}t} \text{ $Q$-a.s.\ as }t\to\infty.
  \end{equation*}
  for $t\in\mathbb{N}$.
\end{lem}
\begin{proof}
  \begin{align*}
    & Q\left( \sup_{f\in\Gamma_n} \mathbb{E}^X_x \left[ \left. e^{F(X)-F(f)} \right| E_{f,n} \right] > e^{(1+4c_2M)n\kappa^{1/3}t} \right) \\
        & \qquad \le \sum_{f\in\Gamma_n} Q\left( \mathbb{E}^X_x \left[ \left. e^{F(X)-F(f)} \right| E_{f,n} \right] > e^{(1+4c_2M)n\kappa^{1/3}t} \right) \\
        & \qquad \le \sum_{f\in\Gamma_n} e^{-(1+4c_2M)n\kappa^{1/3}t} \mathbb{E}^Q \mathbb{E}^X_x \left[ \left. e^{F(X)-F(f)} \right| E_{f,n} \right] \\
        & \qquad \le \sum_{f\in\Gamma_n} e^{-(1+4c_2M)n\kappa^{1/3}t} \mathbb{E}^X_x \left[ \left. \mathbb{E}^Q\left[ e^{F(X)-F(f)} \right] \right| E_{f,n} \right] \\
        \intertext{ As $F(X)-F(f)$ is a centered Gaussian, $\mathbb{E}^Q\left[ e^{F(X)-F(f)} \right] = e^{\mathbb{E}^Q\left[ F(X)-F(f) \right]^2/2} = e^{d(X,f)^2/2}$.}
        & \qquad = \sum_{f\in\Gamma_n} e^{-(1+4c_2M)n\kappa^{1/3}t} \mathbb{E}^X_x \left[ \left. e^{d(X,f)^2/2} \right| E_{f,n} \right] \\
        & \qquad \le \sum_{f\in\Gamma_n} e^{-(1+4c_2M)n\kappa^{1/3}t} e^{\epsilon_n^2/2} \\
        \intertext{By the definition of $E_{f,n}$ (\ref{def_Efn}).
            We recall Corollary \ref{Gamma_n_bound} and the definition of $\epsilon_n$ (\ref{C_and_epsilon_def}) to finish the proof.}
        & \qquad \le |\Gamma_n| e^{-(1+4c_2M)n\kappa^{1/3}t} e^{\epsilon_n^2/2} \\
        & \qquad \le c_1 e^{ 2c_2 M n \kappa^{1/3} t} e^{-(1+4c_2M)n\kappa^{1/3}t} e^{\frac{1}{2} n\kappa^{1/3}t} \\
        & \qquad = c_1 e^{-(\frac{1}{2}+2c_2M)n\kappa^{1/3}t}.
  \end{align*}
  Again using (\ref{exp_sum}) this is summable over $n\in\mathbb{N}$ and then over $t\in\mathbb{N}$ so that the Borell-Catelli Lemma completes the proof.
\end{proof}

\begin{lem}\label{ball_exclusion_bound}
  For $n\ge2$ we can choose $M>1$ arbitrarily large such that for $\kappa\in(0,\kappa_0(M))$, $\kappa_0(M)$ a nonnegative decreasing function, we have that
  \begin{equation*}
    \mathbb{P}( d(X,\Gamma_{n-1})>\epsilon_{n-1} ) \le 4 e^{-\frac{M^2}{2}(n-1)\kappa^{1/3}t}.
  \end{equation*}
\end{lem}
\begin{proof}
  For each path $X$ we define the path $g_X$ as the linear interpolation between the points $(0, X(0)), (\frac1\kappa, X(\frac1\kappa)), (\frac2\kappa, X(\frac2\kappa)), (\frac3\kappa, X(\frac3\kappa)), \ldots$.
  The we have
  \begin{align}
    \mathbb{P}( d(X,\Gamma_{n-1})>\epsilon_{n-1} ) & \le \mathbb{P}\left( d(X,g_X) > \frac{\epsilon_{n-1}}{2} \right) + \mathbb{P}\left( d(g_X,\Gamma_{n-1})> \frac{\epsilon_{n-1}}{2} \right) \nonumber\\
        & \le \mathbb{P}\left( d(X,g_X) > \frac{\epsilon_{n-1}}{2} \right) + \mathbb{P}\left( g_X \not\in H_t(C_{n-1}) \right) \label{ball_decomp}
  \end{align}
  as $\Gamma_{n-1}$ is an $\frac{\epsilon_{n-1}}{2}$-net of $H_t(C_n-1)$.
  We bound each of these terms in turn.
  \begin{align*}
    \mathbb{P}\left( d(X,g_X) > \frac{\epsilon_{n-1}}{2} \right) & = \mathbb{P}\left( d(X,g_X)^2 > \frac{\epsilon_{n-1}^2}{4} \right) \\
        & = \mathbb{P}\left( \sum_{i=1}^{\kappa t} d(Y_i,0)^2 > \frac{\epsilon_{n-1}^2}{4} \right)
  \end{align*}
  where $Y_i(s) = X(s+(i-1)\kappa) - g_X(s+(i-1)\kappa)$, $s\in(0,\frac{1}{\kappa})$.
  The $Y_i$ are iid rate $\kappa$ Brownian bridges on $(0,{\kappa})$ and that
  \begin{align*}
    d(Y_i,0)^2 & = \mathbb{E}^Q \left[ \int_0^{1/\kappa} dW_{Y_i(1/\kappa-s)}(s) - \int_0^{1/\kappa} dW_0(s) \right]^2 \\
        & = 2 \left( \frac{1}{\kappa} - \int_0^{1/\kappa} \Gamma(Y_i(1/\kappa-s))ds \right) \\
        & \le \frac{2}{\kappa}.
  \end{align*}
  It follows that $d(Y_i,0)^2$ has a logarithmic moment generating function bounded on $\mathbb{R}$,
  \begin{equation*}
    \Lambda(\lambda) \le \max \left\{ \frac{2\lambda}{\kappa}, 1 \right\},
  \end{equation*}
  and therefore has a good rate function \cite{DeZe} such that
  \begin{equation} \label{good_rate_function}
    \lim_{|x|\to\infty} \frac{\Lambda^*(x)}{|x|} = \infty .
  \end{equation}
  Applying Cram\'er's Theorem we have
  \begin{align*}
    \mathbb{P}\left( d(X,g_X) > \frac{\epsilon_{n-1}}{2} \right) & = \mathbb{P}\left( \sum_{i=1}^{\kappa t} d(Y_i,0)^2 > \frac{\epsilon_{n-1}^2}{4} \right) \\
        & = \mathbb{P}\left( \frac{1}{\kappa t} \sum_{i=1}^{\kappa t} d(Y_i,0)^2 > \frac{\epsilon_{n-1}^2}{4 \kappa t} \right) \\
        & \le 2 \exp \left\{ -\kappa t \Lambda^*\left( \frac{\epsilon_{n-1}^2}{4 \kappa t} \right) \right\} \\
        & = 2 \exp \left\{ -\kappa t \Lambda^*\left( \frac{n-1}{4 \kappa^{2/3}} \right) \right\} \\
        \intertext{ Restricting $\kappa$ to be small, $\frac{n-1}{4 \kappa^{2/3}}$ is ensured to be large for $n\ge2$ so that by (\ref{good_rate_function}) we have}
        & \le 2 \exp \left\{ -\kappa t \left( \frac{M^2}{2} \frac{n-1}{\kappa^{2/3}} \right) \right\} \\
        & = 2 e^{-\frac{M^2}{2}(n-1)\kappa^{1/3}t}.
  \end{align*}

  Turning to the second term of (\ref{ball_decomp}),
  \begin{align*}
    \mathbb{P}\left( g_X \not\in H_t(C_{n-1}) \right) & = \mathbb{P}\left( \| g_X' \|_2 > C_{n-1} \right) \\
        & = \mathbb{P}\left( \| g_X' \|_2^2 > C_{n-1}^2 \right) \\
        & = \mathbb{P}\left( \sum_{i=1}^{\kappa t} \kappa^{-1} \left\| \frac{X(i/\kappa) - X((i-1)/\kappa)}{\kappa^{-1}} \right\|_2^2 > C_{n-1}^2 \right) \\
        & = \mathbb{P}\left( \frac{1}{\kappa t} \sum_{i=1}^{\kappa t} \left\| X(i/\kappa) - X((i-1)/\kappa)\right\|_2^2 > \frac{C_{n-1}^2}{\kappa^2 t} \right)
  \end{align*}
  Observe that $ \left\| X(i/\kappa) - X((i-1)/\kappa)\right\|_2^2 \sim \chi_d^2$ iid and that chi-squared random variables so they have the logarithmic moment generating function
  \begin{equation*}
    \Lambda(\lambda) = \ln \mathbb{E} e^{\lambda \chi_d^2} =
        \left\{ \begin{array}{l l}
                  -\frac{d}{2} \ln( 1 - 2 \lambda ) & \quad 1 - 2 \lambda > 0 \\
                  \infty & \quad \text{otherwise}
                \end{array}\right. ,
  \end{equation*}
  which has Fenchel-Legendre transform
  \begin{equation*}
    \Lambda^*(x) = \sup_\lambda \{ \lambda x - \Lambda(\lambda) \} = \frac{1}{2}(x + d\ln x - d -d\ln d).
  \end{equation*}
  Taking $x$ large we get that $\Lambda^*(x) \ge \frac{x}{2}$.
  Using Cram\'er's Theorem again,
  \begin{align*}
        \mathbb{P}\left( g_X \not\in H_t(C_{n-1}) \right) &\le 2 \exp \left\{ -\kappa t \Lambda^*\left( \frac{C_{n-1}^2}{\kappa^2 t} \right) \right\} \\
        & = 2 \exp \left\{ -\kappa t \Lambda^*\left( \frac{M^2(n-1)^2}{\kappa^{2/3}} \right) \right\} \\
        \intertext{Restricting $\kappa$ to again be small ensures that $\frac{M^2(n-1)^2}{\kappa^{2/3}}$ is large for $n\ge2$ and $M\ge1$ so that}
        & \le 2 \exp \left\{ -\kappa t \left( \frac{M^2(n-1)^2}{2\kappa^{2/3}} \right) \right\} \\
        & = 2 e^{ -\frac{M^2}{2} (n-1)^2 \kappa^{1/3} t } \\
        & \le 2 e^{ -\frac{M^2}{2} (n-1) \kappa^{1/3} t }.
  \end{align*}
Combining these bounds using (\ref{ball_decomp}) completes the lemma.
\end{proof}

Supporting lemmata complete, we return to the proof of Theorem \ref{one_third_bound}.
Beginning with (\ref{decomp1}) we apply Corollaries \ref{Gamma_n_bound}, \ref{F-T_bound} and Lemmata \ref{X-f_bound}, \ref{ball_exclusion_bound}.
\begin{align*}
  & \limsup_{t\to\infty} u(x,t) \\
  & \qquad \le \limsup_{t\to\infty} \sum_{n\ge1} |\Gamma_n| \left( \sup_{f\in\Gamma_n} e^{F(f)} \right) \left( \sup_{f\in\Gamma_n} \mathbb{E}^X_x \left[ \left. e^{F(X)-F(f)} \right| E_{f,n} \right] \right)
            \nonumber\\ & \qquad\qquad \cdot \mathbb{P}( d(X,\Gamma_{n-1})>\epsilon_{n-1} ) \\
  & \qquad \le \limsup_{t\to\infty} \sum_{n\ge1} c_1 e^{ 2c_2M n \kappa^{1/3} t} e^{4 c_4 c_3^{1/2} M^{1/2} n \kappa^{1/3} t} e^{(1+4c_2M)n\kappa^{1/3}t} \mathbb{P}( d(X,\Gamma_{n-1})>\epsilon_{n-1} ) \\
  & \qquad = \limsup_{t\to\infty} \sum_{n\ge1} c_1 e^{ (1 + 6c_2M + 4 c_4 c_3^{1/2} M^{1/2} ) n \kappa^{1/3} t} \mathbb{P}( d(X,\Gamma_{n-1})>\epsilon_{n-1} ) \\
  & \qquad \le \limsup_{t\to\infty} c_1 e^{ (1 + 6c_2M + 4 c_4 c_3^{1/2} M^{1/2} ) \kappa^{1/3} t} \\
      & \qquad\qquad\qquad + \sum_{n\ge2} c_1 e^{ (1 + 6c_2M + 4 c_4 c_3^{1/2} M^{1/2} ) n \kappa^{1/3} t} 4 e^{-\frac{M^2}{2}(n-1)\kappa^{1/3}t} \\
  & \qquad = \limsup_{t\to\infty} c_1 e^{ (1 + 6c_2M + 4 c_4 c_3^{1/2} M^{1/2} ) \kappa^{1/3} t} \\
      & \qquad\qquad\qquad + 4 c_1 e^{\frac{M^2}{2}\kappa^{1/3}t} \sum_{n\ge2} e^{ (1 + 6c_2M + 4 c_4 c_3^{1/2} M^{1/2} -\frac{M^2}{2} ) n \kappa^{1/3} t} \\
  \intertext{Taking $M$ large so that $\frac{M^2}{2} > 1 + 6c_2M 4 + c_4 c_3^{1/2} M^{1/2}$ we can apply (\ref{exp_sum}),}
  & \qquad \le \limsup_{t\to\infty} c_1 e^{ (1 + 6c_2M + 4 c_4 c_3^{1/2} M^{1/2} ) \kappa^{1/3} t} \\
      & \qquad\qquad\qquad + 4 c_1 e^{\frac{M^2}{2}\kappa^{1/3}t} e^{ (1 + 6c_2M + 4 c_4 c_3^{1/2} M^{1/2} -\frac{M^2}{2} ) \kappa^{1/3} t} \\
  & \qquad \le \limsup_{t\to\infty} 5 c_1 e^{ (1 + 6c_2M + 4 c_4 c_3^{1/2} M^{1/2} ) \kappa^{1/3} t}.
\end{align*}
Returning to the definition of $\lambda(\kappa)$ we then have for $\kappa\in(0,\kappa_0(M))$
\begin{align*}
  \lambda(\kappa) & \le \limsup_{t\to\infty} \frac{1}{t} \ln u(x,t) \\
      & \le \limsup_{t\to\infty} \frac{1}{t} \ln 5 c_1 e^{ (1 + 6c_2M + 4 c_4 c_3^{1/2} M^{1/2} ) \kappa^{1/3} t} \\
      & = (1 + 6c_2M + 4 c_4 c_3^{1/2} M^{1/2} ) \kappa^{1/3}
\end{align*}
which completes the proof of Theorem \ref{one_third_bound}.

\bibliography{lyapunov_one_third_bound}

\def\polhk\#1{\setbox0=\hbox{\#1}{{\o}oalign{\hidewidth
  \lower1.5ex\hbox{`}\hidewidth\crcr\unhbox0}}}
\begin{thebibliography}{1}

\bibitem{Ad90}
Robert~J. Adler.
\newblock {\em {An introduction to continuity, extrema, and related topics for
  general {G}aussian processes}}.
\newblock {Institute of Mathematical Statistics Lecture Notes---Monograph
  Series, 12}. Institute of Mathematical Statistics, Hayward, CA, 1990.

\bibitem{CaViMo}
Ren{\'e} Carmona, Frederi~G. Viens, and S.~A. Molchanov.
\newblock {Sharp upper bound on the almost-sure exponential behavior of a
  stochastic parabolic partial differential equation}.
\newblock {\em Random Oper. Stochastic Equations}, 4(1):43--49, 1996.

\bibitem{CaMo}
Ren{\'e}~A. Carmona and S.~A. Molchanov.
\newblock {Parabolic {A}nderson problem and intermittency}.
\newblock {\em Mem. Amer. Math. Soc.}, 108(518):viii+125, 1994.

\bibitem{CaVi}
Ren{\'e}~A. Carmona and Frederi~G. Viens.
\newblock {Almost-sure exponential behavior of a stochastic {A}nderson model
  with continuous space parameter}.
\newblock {\em Stochastics Stochastics Rep.}, 62(3-4):251--273, 1998.

\bibitem{CrMo06}
M.~Cranston and T.~S. Mountford.
\newblock {Lyapunov exponent for the parabolic {A}nderson model in {$\mathbf
  R^d$}}.
\newblock {\em J. Funct. Anal.}, 236(1):78--119, 2006.

\bibitem{CrMoSh}
M.~Cranston, T.~S. Mountford, and T.~Shiga.
\newblock {Lyapunov exponents for the parabolic {A}nderson model}.
\newblock {\em Acta Math. Univ. Comenian. (N.S.)}, 71(2):163--188, 2002.

\bibitem{DeZe}
Amir Dembo and Ofer Zeitouni.
\newblock {\em {Large deviations techniques and applications}}, volume~38 of
  {\em {Stochastic Modelling and Applied Probability}}.
\newblock Springer-Verlag, Berlin, 2010.
\newblock Corrected reprint of the second (1998) edition.

\bibitem{KoTi}
A.~N. Kolmogorov and V.~M. Tihomirov.
\newblock {{$\varepsilon $}-entropy and {$\varepsilon $}-capacity of sets in
  functional space}.
\newblock {\em Amer. Math. Soc. Transl. (2)}, 17:277--364, 1961.

\end{thebibliography}
\bibliographystyle{plain}
\end{document}